\documentclass[a4paper,12pt]{article}
\usepackage{amsmath,amsthm,amsfonts,shuffle}
\usepackage{fullpage,comment,graphicx,epstopdf}
\usepackage{multicol}

\input{xy}
\xyoption{all}
\xyoption{matrix}

\theoremstyle{plain}

\newtheorem{teo}{Theorem}
  \newtheorem{lem}[teo]{Lemma}
  \theoremstyle{remark}
  \newtheorem{rem}[teo]{Remark}
 
\theoremstyle{definition}
  \newtheorem{example}[teo]{Example}
  \theoremstyle{definition}
  \newtheorem{defn}[teo]{Definition}
    \newtheorem{defi}[teo]{Definition}

  \newtheorem{ex}[teo]{Example}
  \theoremstyle{plain}
  \newtheorem{prop}[teo]{Proposition}
  \theoremstyle{plain}
  
  \newtheorem{coro}[teo]{Corollary}

\def\t{\triangleleft}
\def\tt{\t^{-1}}

\def\wt{\widetilde}

\def\End{\mathrm{End}}

\def\id{\mathrm{Id}}
\def\F{\mathbb{F}}
\def\N{\mathbb{N}}

\def\Z{\mathbb{Z}}

\def\ra{\overline}
\def\To{\Rightarrow}

\def\m1{^{ \hbox{\small{-}}1}}
\def\Ur{U_{nc}^{\gamma}}
\def\Unc{U_{nc}}
\def\Un{U_{nc}}

\title{Link and knot invariants from non-abelian Yang-Baxter 
2-cocycles}

\author{Marco A. Farinati\thanks{Member of CONICET. Partially supported by
PIP 11220110100800CO, and UBACYT 20021030100481BA, mfarinat@dm.uba.ar.} 
\ and Juliana Garc\'ia Galofre\thanks{Partially supported by 
PIP 11220110100800CO and UBACYT 20021030100481BA,
jgarciag@dm.uba.ar }
}
\begin{document}
\maketitle
\begin{abstract}
We define a knot/link invariant using set theoretical solutions $(X,\sigma)$ of the Yang-Baxter equation
and non commutative 2-cocycles. We also define, for a given $(X,\sigma)$, a universal group $\Unc(X)$ governing all 2-cocycles in $X$, and we exhibit examples of computations.
\end{abstract}

\section*{Introduction and preliminaries}
The first part of this work consist of a generalization to biquandles and the notion of non-commutative 2-cocycle
given in \cite{AG} for quandles. It is also a generalization to the non-commutative case of part of the work
in \cite{CEGS} for commutative cocycles.
In this way, we obtain in principle new invariants for biquandles that do not come from quandles, admitting
non-commutative 2-cocycles, that is, whose universal group (see section 2 or 3) is non abelian.

In the second section we define a universal group governing all 2-cocycles for a given biquandle $X$, that is, a group
$\Unc(X)$ together with a 2-cocycle $\pi:X\times X\to \Unc(X)$ such that 
if $f:X\times X\to G$ is a non commutative
2-cocycles with values in a group $G$, then there is a unique group homomorphism
$\wt f:\Unc(X)\to G$ such that $f=\wt f\pi$. For instance, if $\Unc(X)$ is the trivial group, then every 2-cocycle is trivial. On the opposite,
if $\Unc(X)$ is nontrivial, this universal property says that it carries all information that any group could give using non commutative 2-cocycles.

In the third section, a reduced version of $\Unc$ is given, it depends on a map $\gamma:X\to \Unc(X)$. The
constructed group is called $\Ur(X)$, in particular it is a group and there is given a 2-cocycle $\pi_\gamma:X\times X\to \Ur(X)$
with the following property
(Theorem \ref{Ured}): if $f:X\times X\to G$ is a 2-cocycle, then there exists a cohomologous
(see definition \ref{cohomolo})
2-cocycle $f_\gamma:X\times X\to G$ and a group homomorphism $\wt f_\gamma$ such that $f_\gamma=
\wt f_\gamma\pi_\gamma$. Since the invariant defined in section 1 is unchanged for cohomologous cocycles
(Proposition \ref{cohomologous}), the invariant
produced with $f$ is the same as the one coming from $f_\gamma$, so we see that
all invariants are governed by the group $\Ur$, which is, in general,  smaller than $\Unc$.

In section 4 we exhibit some examples of computations. 
Most of them were performed using \cite{GAP}. We wish to heartily thank
Leandro Vendramin for teaching us the basic facts on syntax and GAP programing,
and helping us with our first (and not so firsts) steps implementing the algorithms
we needed. An interesting observation on examples is that, if $(X,\sigma)$
is a solution of the Yang-Baxter equation, then also is $ \ra\sigma:=\sigma^{-1}$, and if
$\sigma$ makes $X$ into a biquandle (see definition below), then also $\ra\sigma$ gives a biquandle structure. One may suspect that $\ra\sigma$
is, in a sense, equivalent to $\sigma$ and probably gives no new information, but this is not the case: one may have $\Ur(\sigma)=1$
(so that $\sigma$ gives always trivial invariants for any 
2-cocycle $f$)  but $\ra\sigma$ may give non trivial invariants.
Section 5 end with concluding remarks.

Before going into section 1 we recall the notion of biquandle and quandle:

\begin{defi}
A set theoretical solution of the Yang-Baxter equation is a pair
$(X,\sigma)$ where $\sigma:X\times X\to X\times X$ is a bijection satisfying
\[
( \id\times \sigma)( \sigma\times \id)( \id\times \sigma)
=( \sigma\times \id)( \id\times \sigma)( \sigma\times \id)
\]
Notation:  $\sigma(x,y)=(\sigma^1(x,y),\sigma^2(x,y))$ and
$\sigma^{-1}(x,y)=\overline{\sigma}(x,y)$.

A solution $(X,\sigma)$ is called
non-degenerated, or {\bf birack} if in addition: 
\begin{enumerate}
 \item
 ({\em left invertibility})
 for any $x,z\in X$ there exists a unique $y$ such that $\sigma^1\!(x,y)=z$, 
 \item ({\em right invertibility}) for any $y,t\in X$ there exists a unique $x$ such that $\sigma^2(x,y)=t$.
\end{enumerate}
A birack is called {\bf biquandle} if, given $x_0\in X$, there exists a unique $y_0\in X$ such that
$\sigma(x_0,y_0)=(x_0,y_0)$. In other words, if there exists a bijective map $s:X\to X$ such
that
\[
\{(x,y):\sigma(x,y)=(x,y)\}=
\{(x,s(x)): x\in X\}
\]
\end{defi}

\begin{defi}
 A {\bf rack} is a pair $(X,\t)$ where $\t:X\to X$ verifies
 \begin{enumerate}
  \item for all $x\in X$, the map $-\t x:X\to X$ ($y\mapsto y\t x$) is bijective, and
 \item  $(x\t y)\t z=(x\t z)\t (y\t z)$
 \end{enumerate}
A rack is called a {\bf quandle} if $x\t x=x$ for all $x\in X$.
\end{defi}
Biracks and biquandles generalizes racks and quandles respectively because a map $\sigma:X\times X\to X\times X$ of the form
$\sigma(x,y)=(y,f(x,y))$ is a birack if and only if the operation $\t$ given by  $x\t y:=f(x,y)$ is a rack, and $\sigma$ is a biquandle if and only
if this rack is a quandle. In section 2 we  review some examples of biquandles.

\section{Non-abelian 2-cocycles}
Let $(X,\sigma)$ be a biquandle
and H a (not necessarily abelian) group.

 \begin{defi} 
 \label{nc2}
 A function $f:X\times X\rightarrow H$
is a {\em braid non-commutative 2-cocycle} if 
\begin{itemize}
\item 
$ f\big(x_1,x_2\big)f\big(\sigma^2(x_1,x_2),x_3\big)
 =f\big(x_1,\sigma^1\!(x_2,x_3)\big)f\big(\sigma^2(x_1,\sigma^1\!(x_2,x_3)),\sigma^2(x_2,x_3)\big)
$, and
\item
$f\big(\sigma^1\! (x_1, x_2), \sigma^1\!(\sigma^2(x_1,x_2),x_3)\big)=f\big(x_2,x_3\big)$
\end{itemize} 
are satisfied for any $x_1, x_2,x_3 \in X$.
\end{defi}

\begin{defn}
 If $f$ further satisfies $f(x, s(x))= 1$ for all $x\in X$ then it is called of {\em type I}.
\end{defn}


\begin{rem}
If $f$ is a braided non commutative 2-cocycle and $\lambda:X\to H$ is
 an arbitrary function, then
\[f'(x,y)=\lambda(x)f(x,y)\lambda^{-1}(\sigma^2(x,y))\]
is also a braided non commutative 2-coycle. If moreover $f$ is of type I, and
$\lambda(x)=\lambda(s(x))$ for all $x\in X$,
then $f'$ is also of type I.
\end{rem}

\begin{defi}\label{cohomolo}
Two cocycles $f,f'$ are cohomologous ($f\sim f'$) if there is a function \\
$\gamma: X\rightarrow H$ such that 
$ \gamma(x)=\gamma(s(x))$
and
\[
f'(x,y)=\gamma(x)f(x,y)\gamma^{-1}(\sigma^2(x,y)) ,
\
\forall x,y\in X.
\]
\end{defi}

\begin{rem}It is easy to see that $\sim$ is an equivalence relation.
\end{rem}

An equivalence class is called a cohomology class. The {\em set} of cohomology classes is denoted by
$H^2_{NC}(X,H)$. This definitions, in case $(X,\t)$ is a quandle and considering
$\sigma(x,y)=(y,x\t y)$, agree with the ones in \cite{CEGS},
since in this case the second condition  of definition \ref{nc2}
is trivial. 
As in the rack/quandle case, if
$H$ is not commutative, $H^2_{NC}(X,H)$ need not
to be a group, it is just a set.

 \begin{rem} If $H$ happens to be commutative and
$f:X\times X\to H$ is a 2-cocycle in the non commutative sense, then $f$ is
necessarily a (special type of) 2-cocycle  with trivial coefficients in the
sense of \cite{CES}, but our definition is more restrictive, because we ask for a set of equations of 
the form $ab=a'b'$ and $c=c'$ (plus being type I), while in the usual abelian 2-cocycles the equation is of 
the form $abc=a'b'c'$ (plus being type I).
\end{rem}

\begin{rem}
The first condition  of definition \ref{nc2} is invariant under
inverting $\sigma$, namely, $f$ satisfies it for $\sigma$
if and only if $f$ does it for
$\sigma^{-1}$. On the other hand, the second condition is not
invariant under inverting $\sigma$. For example,
if $(X,\t)$ is a rack and $\sigma(x,y)=(y, x\t y)$,
then the second condition  is trivially satisfied for any function $f$ 
(and hence, this definition is equivalent, in this setting, to the one given in \cite{CEGS}),
while for $\ra\sigma(x,y)=(y\t^{-1} x,x)$ means that $f$ must be invariant under the action of the Inner group
associated to the rack $X$.
\end{rem}

\subsection{Weights}
Let $X$ be a biquandle, $H$ a group, $f:X\times X\to H$ a non-abelian 2-cocyle. Let $L=K_1\cup\dots\cup K_r$ be a classical oriented link
diagram on the plane, where $K_1,\dots , K_r$ 
are connected components, for some positive integer $r$. 
A {\em coloring} of $L$ by $X$ is a rule that assigns an element of $X$ to each semi-arc of $L$, in such a way that
for every  crossing
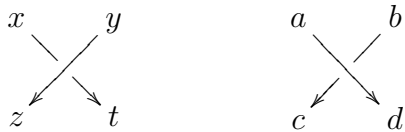
\[
\xymatrix{
x\ar@{->}[rd]|\hole&y\ar[ld]\\
z&t
}
\hskip 2cm 
\xymatrix{
a\ar@{->}[rd]&b\ar[ld]|\hole\\
c&d
}
\]
we have $(z,t)=\sigma(x,y)$ if the crossing is positive, and $(c,d)=\sigma^{-1}(a,b)$ if the crossing
is negative.

Let $\mathcal{C}\in Col_X(L)$ be a coloring of $L$ by $X$ and
 $(b_1,\dots, b_r)$ a set of base points on the components $(K_1,\dots, K_r)$.
Let $\tau^{(i)}$, for $i=1,\dots, r$ the set of crossings such that the under-arc is from the component $i$. 
Let $(\tau_1^{(i)},\dots,\tau_{k_{(i)}}^{(i)})$ be the crossings in  $\tau^{(i)}$, $i=1,\dots,r$ such that appear 
in this order when 
one travels $K_j$ in the given orientation.

At a positive crossing $\tau$, let $x_{\tau}, y_{\tau}$ be the color  on the incoming arcs. 
The {\it Boltzmann weight} at $\tau$ is 
$B_f(\tau, \mathcal{C})=f(x_{\tau}, y_{\tau})$.
At a negative crossing $\tau$, denote $\sigma(x_{\tau}, y_{\tau})$  the colors  on the incoming  arcs.
The {\it Boltzmann weight} at $\tau$ is 
$B_f(\tau, \mathcal{C})=f^{-1}(x_{\tau},y_{\tau})$
\[B_{f,\tau}=f(x_{\tau},y_{\tau}	):
\xymatrix@-2ex{
x_{\tau}\ar@{->}[rd]|\hole&y_{\tau}\ar[ld]\\
\sigma^1\!(x_{\tau}, y_{\tau})&\sigma^2(x_{\tau}y_{\tau})
}
\hskip 0.5cm
\xymatrix@-2ex{
\sigma^1\!(x_{\tau}, y _{\tau})\ar@{->}[rd]&\sigma^2(x_{\tau},y_{\tau})\ar[ld]|\hole\\
x_{\tau}&y_{\tau}
}:B_{f,\tau}=f(x_{\tau},y_{\tau})^{-1}
\]

We will show that a convenient product of these weights is invariant under Reidemeister moves.
 
\subsection{Reidemeister type I moves}
First notice that  $\sigma(x,s(x))=(x,s(x))$ implies $\sigma^{-1}(x,s(x))=(x,s(x))$, so,
adding any  orientation to the diagram
\vskip -1.5cm
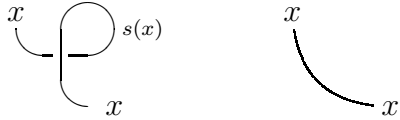
\[    \xymatrix{
&    &\\
x \ar@{-} `d[r] `r[ru]|(.4)\hole `u[rl]_{s(x)}`l[rd]`d[dr] &&&x \ar@{-}@/_1pc/[rd]\\
& x&&&x
    }
  \]
   the condition $f^{\pm 1}(x, s(x))=1$ assures
that the factor due to this crossing
do not count.

\subsection{Reidemeister type II moves}
We consider several cases:
   \begin{itemize}

     \item[Case 1:]
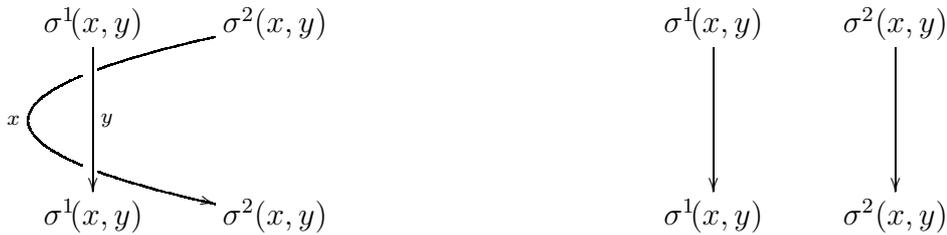
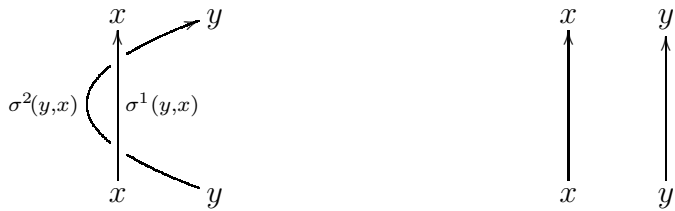
\[\xymatrix{
 \sigma^1\!(x,y)\ar@{->}[dd]^{y}&\ar@{->}@/_7.7pc/[dd]_x|(.25)\hole|(.75)\hole \sigma^2(x,y)\\
 &\\
 \sigma^1\!(x,y)&\sigma^2(x,y)&&&&\\
 }
\xymatrix{
 \sigma^1\!(x,y)\ar@{->}[dd]&\ar@{->}[dd]\sigma^2(x,y)\\
 &\\
 \sigma^1\!(x,y)&\sigma^2(x,y)&&\\
 }
\]

    \item[Case 2:]

\[\xymatrix{
 x\ar@{<-}[dd]^{\sigma^1(y,x)}&\ar@{<-}@/_4pc/[dd]_{\sigma^2\!(y,x)}|(.25)\hole|(.75)\hole y\\
 &\\
 x&y&&&&\\
 }
\xymatrix{
 x\ar@{<-}[dd]&\ar@{<-}[dd] y\\
 &\\
 x&y&&\\
 }
\]

     \item[Case 3:]
 In this case and the following, start naming the top arcs of the diagrams on the left, the rest of the arcs are known as $X$ is a biquandle.
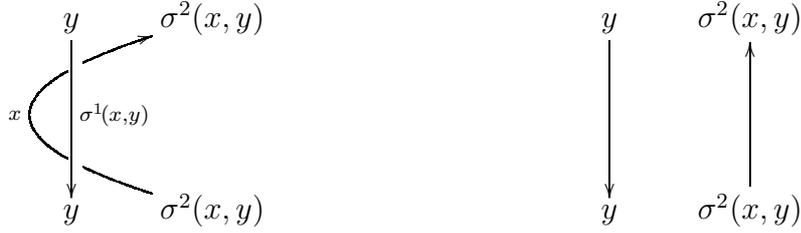
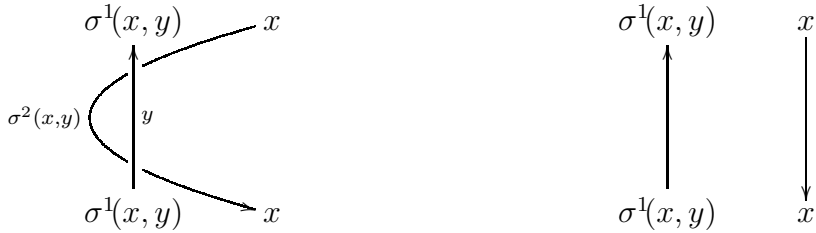
\[\xymatrix{
 y\ar@{->}[dd]^{\sigma^1\!(x,y)}&\ar@{<-}@/_5.7pc/[dd]_x|(.25)\hole|(.75)\hole \sigma^2(x,y)\\
 &\\
 y&\sigma^2(x,y)&&&&\\
 }
\xymatrix{
 y\ar@{->}[dd]&\ar@{<-}[dd]\sigma^2(x,y)\\
 &\\
y&\sigma^2(x,y)&&\\
 }
\]

 \item[Case 4:]

\[\xymatrix{
 \sigma^1\!(x,y)\ar@{<-}[dd]^{y}&\ar@{->}@/_5.7pc/[dd]_{\sigma^2(x,y)}|(.25)\hole|(.75)\hole x\\
 &\\
 \sigma^1\!(x,y)&x&&&&\\
 }
\xymatrix{
 \sigma^1\!(x,y)\ar@{<-}[dd]&\ar@{->}[dd]x\\
 &\\
\sigma^1\!(x,y)&x&&\\
 }
\]

\end{itemize}
The product of weights corresponding to the diagrams on the left in cases 1 and 3  is $f^{-1}(x,y)f(x,y)=1$, in cases 2 and 4 
 is $f(x,y)f^{-1}(x,y)=1$.

 \subsection{Reidemeister type III moves}
 
While there are eight oriented Reidemeister type III moves, only four of them are different.

\

{\em Case 1:}
Start by naming the incoming-arcs $x_1,x_2,x_3$. 
In case 1, as well as in the rest of the cases, once chosen three arcs in both diagrams the remaining
arcs are respectively equal as $\sigma$ is  a solution of YBeq.
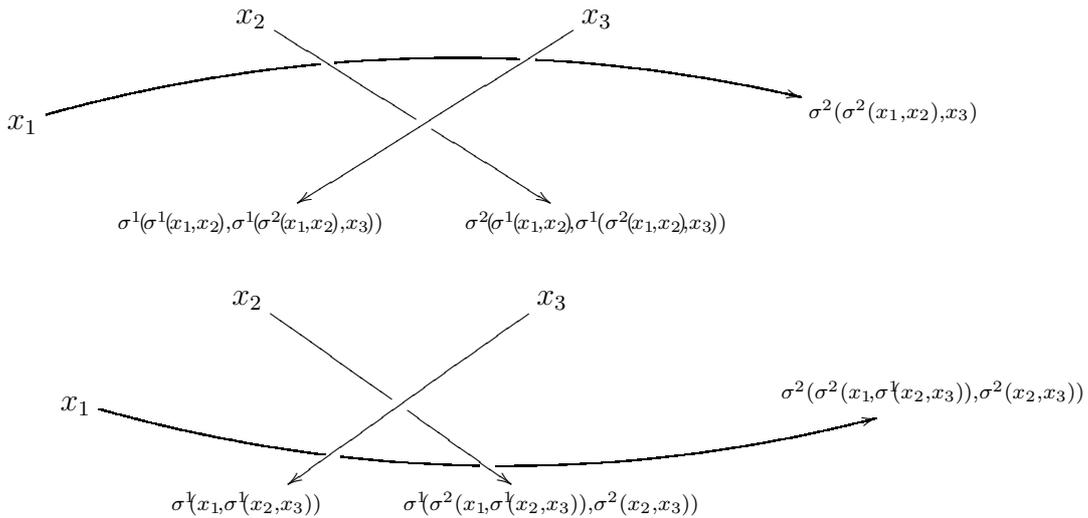
\[
  \xymatrix{
  &x_2\ar@{->}[rdd]|\hole&\ar@{->}[ldd]x_3&\\
x_1\ar@{->}@/^2pc/[rrr]|(.35)\hole|(.58)\hole&&&{}^{\sigma^2(\sigma^2(x_1,x_2),x_3)}\\  
&{}^{\sigma^1\!(\!\sigma^1\!(\!x_1\!,x_2\!),\sigma^1\!(\!\sigma^2\!(\!x_1\!,x_2\!),x_3\!))}
&{}^{\sigma^2\!(\!\sigma^1\!(\!x_1\!,x_2\!)\!,\sigma^1\!(\sigma^2\!(\!x_1\!,x_2\!)\!,x_3\!))}&&&\\ } 
\]
\[\xymatrix{
  &x_2\ar@{->}[rdd]|\hole&\ar@{->}[ldd]x_3&\\
x_1\ar@{->}@/_2pc/[rrr]|(.3)\hole|(.48)\hole&&&{}^{\sigma^2(\sigma^2(x_1\!,\sigma^1\!\!(x_2,x_3)),\sigma^2(x_2,x_3))}\\  
&{}^{\sigma^1\!\!(x_1\!,\sigma^1\!\!(x_2,x_3))}&{}^{\sigma^1\!\!(\sigma^2(x_1\!,\sigma^1\!\!(x_2,x_3)),\sigma^2(x_2,x_3))}&\\ } 
\]
The product of the weights following the horizontal under-arc, in the first diagram, is: 
\[
I=f(x_1, x_2)f(\sigma^2(x_1,x_2),x_3)
\]
and in the second, is: 
\[
 II=f(x_1, \sigma^1\!(x_2,x_3))f(\sigma^2(x_1,\sigma^1\!(x_2,x_3)),\sigma^2(x_2,x_3))
\]

$I=II$ is one of the equalities defining 2-cocycle, the
other equation defining 2-cocycle affirms that the weights given to the {\em other} crossings are the same.
Notice that in the quandle coloring this condition is trivial, but in the biquandle coloring it is not.

 \
 
{\em Case 2:} 
 Start by naming the arcs $\sigma^2(x_1,x_2)$, $x_2$ and $\sigma^1\!(x_2,x_3)$ in both diagrams. 
 The remaining arcs are known using  
 the  fact that $X$ is a biquandle
 and \\$\sigma^1\left(\sigma^1(x_1,x_2),\sigma^1(\sigma^2(x_1,x_2),x_3)\right)
 =\sigma^1(x_1,\sigma^1(x_2,x_3))$ (due to the braid equation).

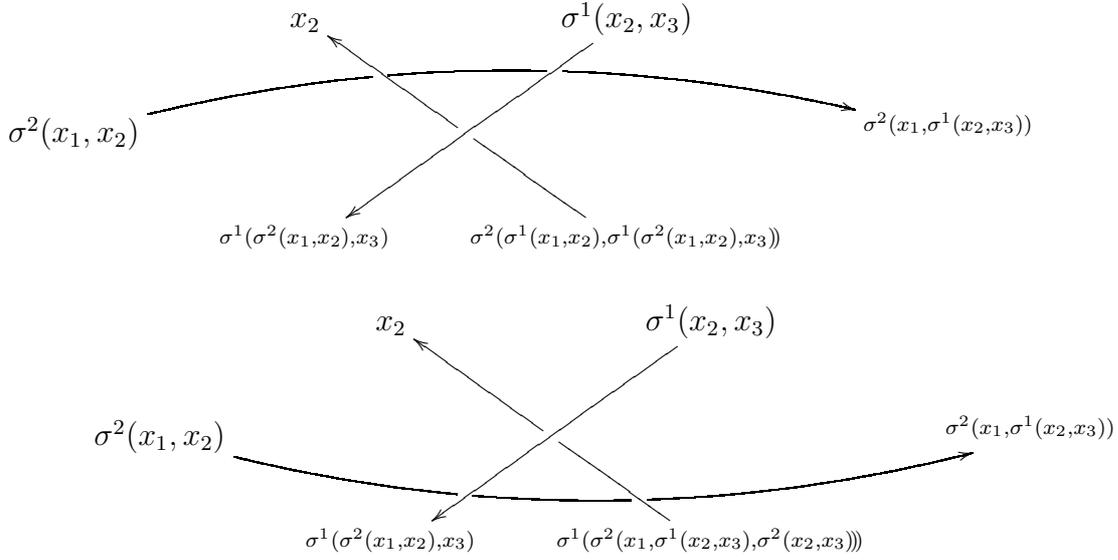
\[
  \xymatrix{
  &x_2\ar@{<-}[rdd]|\hole&\ar@{->}[ldd]\sigma^1(x_2,x_3)&\\
\sigma^2(x_1,x_2)\ar@{->}@/^2pc/[rrr]|(.35)\hole|(.55)\hole&&&{}^{\sigma^2(x_1,\sigma^1(x_2,x_3))}\\  
&{}^{\sigma^1(\sigma^2(x_1,x_2),x_3)}&{}^{\sigma^2(\sigma^1(x_1,x_2),\sigma^1(\sigma^2(x_1,x_2),x_3)\!)}&&&&\\ } 
\]
\[
\xymatrix{
  &x_2\ar@{<-}[rdd]|\hole&\ar@{->}[ldd]\sigma^1(x_2,x_3)&\\
\sigma^2(x_1,x_2)\ar@{->}@/_2pc/[rrr]|(.35)\hole|(.55)\hole&&&{}^{\sigma^2(x_1,\sigma^1(x_2,x_3))}\\  
&{}^{\sigma^1(\sigma^2(x_1,x_2),x_3)}&{}^{\sigma^1(\sigma^2(x_1,\sigma^1(x_2,x_3),\sigma^2(x_2,x_3)\!)\!)}&\\ } 
\]

 The product of weights for the horizontal line in the first diagram is \[I=f^{-1}(x_1,x_2)f(x_1,\sigma^1\!(x_2,x_3))\]
 and for the second 
 diagram is 
 \[
 II=f(\sigma^2(x_1,x_2),x_3)f^{-1}(\sigma^2(x_1,\sigma^1\!(x_2,x_3)),\sigma^2(x_2,x_3))
 \]
 The remaining weights in both diagrams are  
 $a=f^{-1}(\sigma^1\!(x_1,x_2),\sigma^1\!(\sigma^2(x_1,x_2),x_3))$
 and 
$b=f^{-1}(x_2,x_3)$. 
As $f$ is a 2-cocycle, $a=b$. 
 
 \
 
 {\em Case 3:} Name the incoming arcs by $a,b$ and $c$. 
 
 \begin{rem} YBeq is equivalent to the following equation, which explains the equality of the out-coming arcs in both diagrams.
\begin{equation}\label{trenza-inversa}
   (\sigma\times 1)(1\times\overline{\sigma})(\overline{\sigma}\times1)=(1\times\overline{\sigma})(\overline{\sigma}\times1)(1\times\sigma)
\end{equation}

 \end{rem}

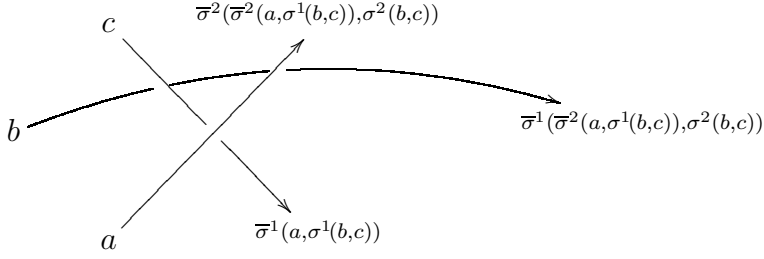
 \[
  \xymatrix{
  &c\ar@{->}[rdd]|\hole&\ar@{<-}[ldd]{}^{\overline{\sigma}^2(\overline{\sigma}^{2}(a,\sigma^1\!(b,c)),\sigma^2(b,c))}&\\
b\ar@{->}@/^2pc/[rrr]|(.235 )\hole|(.42)\hole&&&{}^{\overline{\sigma}^1(\overline{\sigma}^2(a,\sigma^1\!(b,c)),\sigma^2(b,c))}\\  
&a&{}^{\overline{\sigma}^1(a,\sigma^1\!(b,c))}&&&&\\ 
}
\]
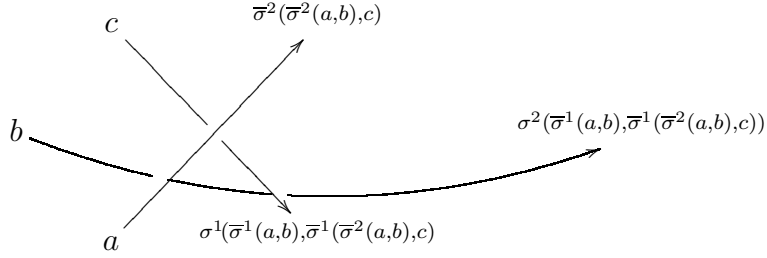
\[
\xymatrix{
  &c\ar@{->}[rdd]|\hole&\ar@{<-}[ldd]{}^{\overline{\sigma}^2(\overline{\sigma}^2(a,b),c)}&\\
b\ar@{->}@/_2pc/[rrr]|(.23)\hole|(.42)\hole&&&{}^{\sigma^2(\overline{\sigma}^1(a,b),\overline{\sigma}^1(\overline{\sigma}^2(a,b),c))}\\  
&a&{}^{\sigma^1\!(\overline{\sigma}^1(a,b),\overline{\sigma}^1(\overline{\sigma}^2(a,b),c)}&\\ } 
\]

 
 The product of weights for the horizontal line in the first diagram is
 \[I=
 f(b,c)f^{-1}\left(\overline{\sigma}(\overline{\sigma}^2 (a,\sigma^1\!(b,c)),
 \sigma^2(b,c))\right)
\] and for the second diagram is 
\[II= f^{-1}(\overline{\sigma}(a,b))f\left(\overline{\sigma}^1(a,b),\overline{\sigma}^1(\overline{\sigma}^2(a,b),c)\right)                                                                    
 .\]
 Using (\ref{trenza-inversa}) in I:
 \[
 I=f(b,c) f\left((\sigma^2(\overline{\sigma}^1(a,b),\overline{\sigma}^1(\overline{\sigma}^2(a,b),c)),
 \overline{\sigma}^2(\overline{\sigma}^2(a,b),c))\right)
 \]
  Take te changes of variables  
   $(x_1,d)= \overline{\sigma}(a,b)$ and $(x_2,x_3)=\overline{\sigma}(d,c)$. Then
   \[
    I=f(\sigma^2(x_1,\sigma^1(x_2,x_3)),\sigma^2(x_2,x_3))f^{-1}(\sigma^2(x_1, x_2),x_3)
   \]
\[
 II=f^{-1}(x_1,\sigma^1(x_2,x_3))f(x_1,x_2)
\]
We see that if  $f$ is a non-commutative 2 cocycle then $I=II$.
  
  The weights that correspond to the other crossings are: 
  \[
   III=f^{-1}(\overline{\sigma}(a, \sigma^1(b,c)))
,  \
     IV=f^{-1}(\overline{\sigma}(\overline{\sigma}^2(a,b),c)),
    \]
changing variables and composing (\ref{trenza-inversa}) with $1\times \delta$:

\[
 III=f^{-1}(\sigma^1(x_1,x_2),\sigma^1(\sigma^2(x_1,x_2),x_3))
,\
 IV=f^{-1}(x_2,x_3)
\]

$III=IV$ is verified as $f$ is a 2-cocycle.

 \
 
 {\em Case 4:} We only exhibit the diagram corresponding to this case, the computations are similar to the previous case.

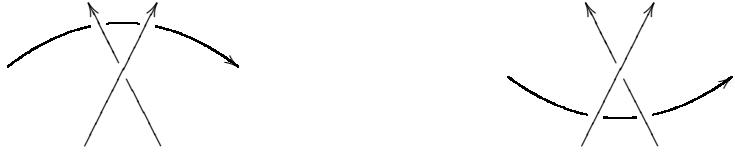
\[
  \xymatrix{
  &\ar@{<-}[rdd]|\hole&\ar@{<-}[ldd]&\\
\ar@{->}@/^1.5pc/[rrr]|(.4)\hole|(.6)\hole&&&\\  
&&&&&&\\ } 
\xymatrix{
  &\ar@{<-}[rdd]|\hole&\ar@{<-}[ldd]&\\
\ar@{->}@/_1.5pc/[rrr]|(.4)\hole|(.6)\hole&&&\\  
&&&\\ } 
\]

 This shows, not only, that the product of the weights does not change  under Reidemeister moves 
 but the remaining weights stay the same.

For a group element $h\in H$, denote $[h]$ denote the conjugacy class to which $h$ belongs.
\begin{defn}
The set of conjugacy classes  
 \[
  \overrightarrow{\Psi}(L,f)= \overrightarrow{\Psi}_{(X,f)}(L)
  = \{[\Psi_i(L,\mathcal{C},f )]\}_{\underset{\mathcal{C}\in Col_X(L)}{1\leq i\leq r}}
 \]
where  $
        \Psi_i(L,\mathcal{C},f)=\prod^{k(i)}_{j=1}B_f(\tau^{(i)}_j, \mathcal{C})
       $
(the order in this product is following the orientation of the component)
is called the   {\em conjugacy biquandle cocycle invariant} of the link.   
\end{defn}
\begin{teo}
The   conjugacy biquandle cocycle invariant    $\Psi$ is well defined.
\end{teo}
\begin{proof}
 The fact that $\Psi$ does not change under Reidemeister moves for fixed base points was proven earlier. A change of base points
 causes cyclic permutations of Boltzmann weights, and hence the invariant is defined up to conjugacy.
\end{proof}

 \begin{prop}\label{cohomologous}
 If  $f, g$ are two  cohomologous non-commutative 2-cocyle  functions then  $[\Psi_i(L,\mathcal{C},f)]
 =[\Psi_i(L,\mathcal{C},g)]$.
  \end{prop}
\begin{proof} Let us suppose $f(x_1,x_2)=\gamma(x_1)g(x_1,x_2)\gamma^{-1}(\sigma^2(x_1,x_2))$. There are four cases to analyse: 

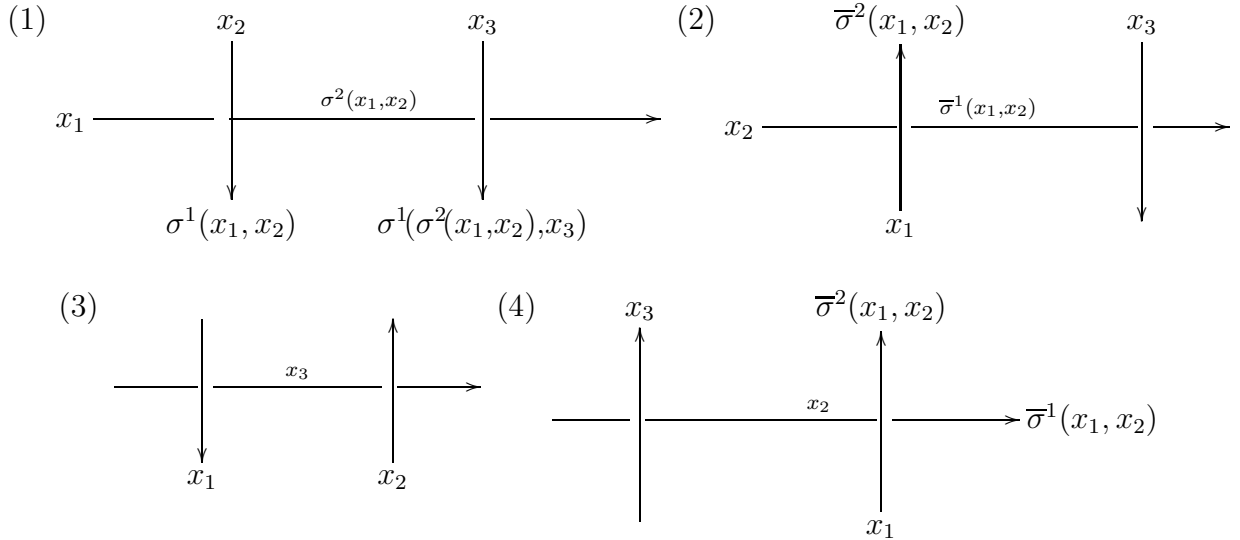
\[   (1)
  \xymatrix{
& x_2 \ar@{->}[dd] & x_3 \ar@{->}[dd]&\\
x_1\ar@{->}[rrr]|(0.253)\hole|(0.6895)\hole^{\sigma^2(x_1,x_2)}&&&\\
&\sigma^1(x_1,x_2)&\sigma^1\!(\sigma^2\!(x_1,\!x_2),\!x_3)&
 }
   (2) \xymatrix{
& \overline{\sigma}^2(x_1,x_2) \ar@{<-}[dd]&& x_3 \ar@{->}[dd]&\\
x_2\ar@{->}[rrrr]|(0.33)\hole|(0.815)\hole^{\ra\sigma^1(x_1,x_2)}&&&&\\
&x_1&&&
 }\]
 \[
(3)
\xymatrix{
&  \ar@{->}[dd] &&  \ar@{<-}[dd]&\\
\ar@{->}[rrrr]|(.26)\hole|(.74)\hole^{x_3}&&&&\\
&x_1&&x_2&
 }
 (4) \xymatrix{
& x_3 \ar@{<-}[dd] && \ra\sigma^2(x_1,x_2) \ar@{<-}[dd]&\\
\ar@{->}[rrrr]|(.17)\hole|(.62)\hole^{x_2}&&&&\ra\sigma^1(x_1,x_2)\\
&&&x_1&
 }\]
In case 1),                   
the product of weights for the horizontal line is:
$
 f(x_1,x_2)f(\sigma^2(x_1,x_2),x_3)=
 $
 \[=\gamma(x_1)g(x_1,x_2)\gamma^{-1}(\sigma^2(x_1,x_2))
 \gamma(\sigma^2(x_1,x_2)g(\sigma^2(x_1,x_2),x_3)\gamma^{-1}(\sigma^2(\sigma^2(x_1,x_2),x_3))
\]
In case 2):
 $
  f^{-1}(\overline{\sigma}(x_1,x_2))f(\overline{\sigma}^1(x_1,x_2),x_3)=
 $
  \[
   \gamma(\sigma^2(\overline{\sigma}(x_1,x_2))) g^{-1}(\overline{\sigma}(x_1,x_2))
   \gamma^{-1}(\overline{\sigma}^{1}(x_1,x_2))
   \gamma(\ra\sigma^{1}(x_1,x_2))
   g(\overline{\sigma}^1(x_1,x_2),x_3)\gamma^{-1}(\sigma^2(\overline{\sigma}^1(x_1,x_2),x_3))
  \]
In case 3):
 $
  f(\ra\sigma(x_1,x_3))f^{-1}(\ra\sigma(x_2,x_3))=
 $
\[
 \gamma(\ra\sigma^1(x_1,x_2))g(\ra\sigma(x_1,x_3))\gamma^{-1}(\sigma^2(\ra\sigma(x_1,x_3)))\gamma(\sigma^2(\ra\sigma(x_2,x_3)))
 g(\ra\sigma(x_2,x_3))\gamma^{-1}(\ra\sigma(x_2,x_3))
\]
And finally 
in case 4):
 $
  f^{-1}(x_2,x_3)f^{-1}(\overline{\sigma}(x_1,x_2))=
 $
 \[
  \gamma(\sigma^2(x_2,x_3))g^{-1}(x_2,x_3)\gamma^{-1}(x_2)  
  \gamma(\sigma^2(\overline{\sigma}(x_1,x_2)))g^{-1}(\overline{\sigma}(x_1,x_2))\gamma^{-1}(\overline{\sigma}^1(x_1,x_2))
 \]

 \end{proof}

\section{Universal non commutative 2-cocycle}

Given a biquandle $(X,\sigma)$ and a group $H$, recall a non commutative 2-cocycle is a function
$f:X\times X\to H$ satisfying
\[
 f\big(x,y\big)f\big(\sigma^2(x,y),z\big)=f\big(x,\sigma^1\!(y,z)\big)f\big(\sigma^2(x,\sigma^1\!(y,z)),\sigma^2(y,z)\big)
\]
and
 \[
     f\big(\sigma^1\! (x, y), \sigma^1\!(\sigma^2(x,y),z)\big)=f\big(y,z\big)
    \]
for any $x,y,z \in X$, and is called type I if in addition
$f(x,s(x))=1$.

\begin{defi}\label{def:unc}
We define $U_{nc}=U_{nc}(X,\sigma)$, the Universal biquandle 2-cocycle group, as the group freely generated by symbols
$(x,y)\in X\times X$ with relations

\begin{enumerate}
\item[(Unc1)] $
(x,y)(\sigma^2(x,y),z)=(x,\sigma^1(y,z))(\sigma^2(x,\sigma^1(y,z)),\sigma^2(y,z))
$
\item[(Unc2)] $
(\sigma^1 (x, y), \sigma^1(\sigma^2(x,y),z))=(y,z)
    $
\item[(Unc3)]
$
 (x,s(x))=1
$
\end{enumerate}
\end{defi}

The following is immediate from the definitions:
\begin{prop}\label{universal property}
Let $X$ be a biquandle:
\begin{itemize}
\item Denote $[x,y]$ the class
of $(x,y) $ in $U_{nc}$. The map
\[
\begin{array}{rcl}
\pi\colon X\times X&\to& U_{nc}
\\
(x,y)&\mapsto &[x,y]
\end{array}
\]
is a type I non commutative 2-cocycle.

\item 
Let $H$ be  a group and
 $f:X\times X\to H$ a type I non commutative 2-cocycle, then there exists a unique group homomorphism
$\ra f:U_{nc}\to H$ such that
$f=\ra f\pi$.
 \[
 \xymatrix{
 X\times X\ar[d]_{\pi}\ar[r]^f&H\\
 U_{nc}\ar@{-->}[ru]_{\exists !\ \ra f}
 }\]
       \end{itemize}
\end{prop}

In particular, given $(X,\sigma)$, there exists non trivial 
2-cocycles if and only if $U_{nc}$ is a nontrivial group.

\begin{prop} $U_{nc}$ is functorial. That is, 
 if $\phi:(X,\sigma)\to (Y,\tau)$ is a morphism of set theoretical solutions of the YBeq, namely $\phi$
satisfy
\[
(\phi\times\phi)\sigma(x,x')=
\tau(\phi x,\phi x')
\]
then, $\phi$ induces a (unique) group homomorphism
$U_{nc}(X)\to U_{nc}(Y)$
satisfying
\[
[x,x']\mapsto
[\phi x,\phi x']
\]
\end{prop}
\begin{proof}
One need to prove that the assignment $(x,x')\mapsto
(\phi x,\phi x')$
is compatible with the relations defining $U_{nc}(X)$ and $U_{nc}(Y)$ respectively,
and that is clear since $(\phi\times\phi)\circ\sigma=\tau\circ (\phi\times\phi)$.
\end{proof}

\begin{rem}
 In order to produce an invariant of a knot or link, given a solution $(X,\sigma)$, we need to produce a coloring
 of the knot/link by $X$, and then find a non commutative 2-cocycle, but since $U_{nc}$ is functorial,
 given $X$ we always have the universal 2-cocycle $X\times X\to U_{nc}(X)$, and hence, we only need to consider all different colorings.
 
 Also, if $\phi:X\to X$ is a bijection commuting with $\sigma$, then, given a coloring and its invariant calculated
 with the universal cocycle, we may apply $\phi$
 to each color and get another coloring, and this will produce the same invariant pushed by $\phi$ in $U_{nc}$.
 
\end{rem}

\begin{rem}Given a link $L$ of two strands colored using (both) colors $\{1,2\}$, the invariant
obtained is 
$\Psi_i(L,\mathcal{C},f)=(i,j)^{ln(i,j)}$ 
$i,j\in \{1,2\}$ and $i\neq j$, where $ln(i,j)$ is the linking number between the two strands.
\end{rem}

\begin{proof}
First notice that every component must be colored by a single  color. 



To every underarc of the component $i$ with itself will correspond a $(i,i)=1$ as weight. Then these crossings will not change
the product. Then one can think that each component is unknoted with itself. 
It is well known that any two closed curves in space, if allowed to pass through themselves but not each other, 
can be moved a concatenation  of the following standard positions:
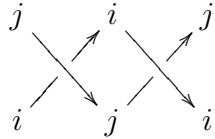
\[
\xymatrix{
j\ar[rd]&i\ar[rd]&j\\
 i\ar[ru]|\hole&j\ar[ru]|\hole&i
}
\]
This diagram will contribute a factor $(i,j)^{\bf 1}$ to $\Psi_i$ ($(j,i)$ to $\Psi_j$)
and if trying to calculate the linking number will 
add {\bf 1} for each pair of crossings. Analogously in the next diagram:
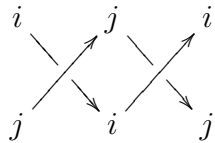
\[
\xymatrix{
i\ar[rd]|\hole&j\ar[rd]|\hole&i\\
j\ar[ru]&i\ar[ru]&j
}
\]
so, the invariant $\Psi_i$ will be
$(i,j)^{\frac{a-b}{2}}=(i,j)^{ln(i,j)}$ where $a, b$ are the total amount of positive and negative crossings
\end{proof}

\begin{example}
 The Whitehead link have linking number cero, the same happens taking the link consisting of two unknots. 
 If you paint these links using $Wada(\Z_3)$
  (see example below), Whitehead
 has only 3 possibilities, while there are 9 ways to paint the pair of unknots.
 \end{example}

\subsection{Some examples of biquandles of small cardinality}

We first list some well-known general constructions generating biquandle solutions:

\begin{enumerate}
\item 
If $(X,\t)$ is a rack, one may consider two different solutions of the YBeq:
\[
\sigma(x,y)=(y,x\t y),\ \hbox{ and }\
\ra\sigma(x,y):=\sigma^{-1}(x,y)=(y\tt x, x)
\]
these solutions are biquandles if and only if $(X,\t)$ is a quandle, namely
$x\t x=x$ for all $x\in X$, in this case the function $s$ is the identity: $s(x)=x$.

When considering n.c. 2-cocycles, 
condition Unc2 is not preserved 
(in general) if one changes $\sigma$ with $\sigma^{-1}$, so it is relevant
to see $ \sigma$ and $\sigma^{-1}$ as different biquandles.

\item Let  $\tau:X\times X\to X\times X$ denote the flip,
namely 
$\tau(x,y)=(y,x)$.
Let  $\mu,\nu:X\to X$  be two bijections of $X$. then
 \[
 (\mu\times \nu) \tau(x,y)
 =(\mu(y),\nu(x)) 
 \]
satisfies YBeq if and only if $\mu\nu=\nu\mu$, 
and this solution is a biquandle
if and only if $\nu=\mu^{-1}$, in this case, the function $s:X\to X$ is equal to  $\mu^{-1}$.
In this way, the set of bijections of $X$ maps injectively into the set of biquandle structures on $X$, 
each conjugacy class of a given bijection  maps into an 
isomorphism class of biquandle structures. Notice that every biquandle structure obtained in this
way is involutive, namely $\sigma= (\mu\times \mu^{-1}) \circ \tau$ verifies $ \sigma^2=\id$.

\item Wada: if $G$ is a group, then the formula $ \sigma(x,y)=(xy^{-1}x^{-1},xy^2)
$
is a biquandle, with $s(x)=x^{-1}$.
As a particular case, if $G$ is abelian and with additive notation we have
$ \sigma(x,y)=(-y,x+2y)
$.
\item Alexander biquandle or Alexander switch:

Let $R$ be a ring, $s,t\in R$ two commuting units, and $M$ an $R$-module, then
\[
 \sigma(x,y)=(s\cdot y,t\cdot  x+(1-st)\cdot y) ,\ \ (x,y)\in M\times M
 \]
is a biquandle, with function $s(x)=(s^{-1})\cdot x$.
In the particular case $s=-1$, $t=1$ one gets the abelian 
Wada's solution. If $s=1$ then one gets the solution induced by the Alexander rack.
 
\end{enumerate}

These constructions give a lot of examples, but there are much more.
If $|X|=2$, call $X=\{0,1\}$, one have the flip, satisfying $s(1)=1$ and $s(2)=2$, and this condition
fully characterize this solution. If $s(0)\neq 0$ then $s(0)=1$ and necessarily $s(1)=0$; this forces 
$ \sigma(0,0)=(1,1)$ and $\sigma(1,1)=(0,0)$. This is actually a biquandle coming from
the bijection construction 
\[
 \sigma(x,y)=(y+1,x-1) : x,y\in \Z/2\Z
\]
We will call this solution  the {\bf antiflip}.

\

If $|X|=3$, we call the elements $X=\{0,1,2\}$ and identify $X=\Z/3\Z$. The
above constructions give the following list:

\begin{enumerate}
\item There are three isomorphism classes of quandles of 3 elements:
\begin{enumerate}
 \item the trivial quandle ($x\t y=x$ for all $x,y$), this gives the flip solution.
\item $D_3$: $x\t y=2y-x$, for $x,y\in \Z/3\Z$, which gives two solutions
 \[
  \sigma(x,y)=(y,x\t y)=(y,2y-x)
 \]
and its inverse
 \[
  \ra\sigma(x,y)=(x\t^{-1} y,x)=(2x-y,x)
 \]
\item another quandle which we call $Q_3$, with operation given by 
 $-\t 0=(12)$ (the permutation $1\leftrightarrow 2$),
 and  $-\t 1= -\t 2=\id$.
 The solution 
 \[
  \sigma(x,y)=(y,x\t y)
 \]
 behaves like the flip for $x,y=1,2$, but
 \[
\sigma(0,1)=(1,0),\ \sigma(1,0)=(0,2)
\]
 \[
\sigma(0,2)=(2,0),\ \sigma(2,0)=(0,1)
\]
One can check that this equalities can be achieved with the formula
\[
 \sigma(x,y)=(y,-x-xy^2)=
(y,-x(1+y^2)) : x,y\in \Z/3\Z
\]
We also have the inverse solution.
\end{enumerate}

\item  If $X=\{0\}\coprod\{1,2\}$, with 
$\sigma(0,i)=(i,0)$ and
$\sigma(i,0)=(0,i)$, then the flip on $\{1,2\}$ produces again the flip on three elements,
but the other solution produce a new solution of the YBeq:

\[\sigma(1,2)=(1,2), \ \sigma(2,1)=(2,1)\]
\[\sigma(1,1)=(2,2),\ \sigma(2,2)=(1,1)\]
\[\sigma(0,i)=(i,0),\ \sigma(i,0)=(0,i)\]
One may check that this equalities are given by
$\sigma(x,y)=(y+x^2y,x+y^2x)$.

\item Wada's construction for $\Z_3$ gives the example
$\sigma(x,y)=(-y,x-y)$ and its inverse: $\ra \sigma(x,y)=(y-x,-x)$.

\item Bijection biquandles:
\begin{enumerate}
\item Using the bijection $\mu(x)=-x$ we have the solution $ \sigma(x,y)=(-y,-x)$,
\item if $\mu(x)=x+1$ then we have the solution $\sigma(x,y)=(y+1,x-1)$. 
One can check that all  bijections $\neq \id$ are conjugated to one of these.

\end{enumerate}
\end{enumerate}
For
$M=R=\Z_3$,  the units of $R$ are $\pm 1$:
 the Alexander biquandle gives Wada's for $s=t=-1$, the Dihedral quandle solution for
$s=1$ and $t=-1$, the flip for $s=t=1$, and the bijection solution
$ \sigma(x,y)=(-y,-x)$ when $s=t=-1$, so we have no new solution in this small cardinality
considering the Alexander biquandle.

In this way, we obtain 10 solutions of the YBeq that are biquandles,
three of them (flip, $Q_3$ and $D_3$) are quandle solutions.
In A. Bartholomew and R. Fenn's classification list (see \cite{BF})
there are 7 biquandles that are not quandles,
but we remark that, 
in Bartholomew and Fenn's list, if a solution $\sigma$
is listed, then $\ra\sigma$ is not listed, even thought
$\sigma$ may not be isomorphic to $\ra\sigma$, as solution
of the Yang-Baxter equation. For instance, for every quandle
$(X,\t)$, the solution with ``name'' $X$ is
$\sigma(x,y)=(y,x\t y)$
but the inverse solution
 $\sigma(x,y)=(x\t^{-1} y,x)$
do not appear in the list.
We use the notation $BQ^3_i$, $i=1,\dots ,10$ for the biquandles solutions
of Bartholomew and Fenn, and we denote
$BQ^{3*}_i$ the inverse solution with respect to $BQ^3_i$.

\subsection{Computations of $U_{nc}$}

We begin with an  explicit computation of the group $U_{nc}$
for a particular example;
Wada: $ \sigma(x,y)=(-y,x-y)$, which is also
BiAlexander with $M=\Z/3\Z, s=-1,t=1$.
 
The fixed points are $(x,s(x))=(x,-x)$, that is $(0,0),(1,2),(2,1)$, so $(0,0)=(1,2)=(2,1)=1$.
Conditions (Unc1-2) are:
 \[
(x,y)(x-y,z)=(x,-z)(x+z,y-z)
\]
 \[
(-y,-z)=(y,z)
    \]    
From the second equality we get generators $a=(0,1)=(0,-1)$,
$b=(1,0)=(-1,0)$ and $c=(1,1)=(-1,-1)$.

From the first cocycle equation
 \[
(x,y)(x-y,z)=(x,-z)(x+z,y-z)
\]
if $y=sx$ (using $(x,sx)=1$ and also using Unc2) we get a trivial equality. If
 $y=sz$    we also get trivial equality, so in principle we have
 $3^3 - 9 - 9 +3=12$ equations.
  If we write them {\em all} in terms of $a$, $b$ and $c$ we get
  
\[
1=1,\ a=a,\ ab=a,\ ac=ac, \ b^2=b^2,\ bc=c,
\]
\[
 1=1, \ c=c,\ b=b,\ 1=1,\ ca=ca,\  ca=b
\]
of course we can see trivial equations, if we exclude them we get
\[
 ab=a,\ bc=c,\  ca=b
\]
from the first (and also the second) equation we can see that te generator $b$ is trivial.  If we write again all the equations with the replacement $b=1$ we get
\[
 ca=1
\]
We conclude $\Unc=Free(a,c)/(ac=1)\cong Free(a)$, but also we have
described a procedure that can be implemented in a computer program:

\begin{enumerate}
\item Add to the set $X\times X$ a new element ''1''  and begin to define
 an equivalence relation $(x,s(x))\sim 1$.
\item from the second condition, add
$(y,z)\sim(\sigma^1 (x, y), \sigma^1(\sigma^2(x,y),z))$ to the 
equivalence relation.

More precisely, given a list of subsets of $(X\times X)\coprod \{1\}$
 whose union is  $(X\times X)\coprod \{1\}$
(if this is not the case we add  the sets $\{(x,y)\}$ to the list) one can easily
give an algorithm producing the partition of $(X\times X)\cup\{1\}$ corresponding
to the equivalence relation generated by the list of subsets:
 for each pair of subsets of the list, with nontrivial
intersection, we replace these two subsets by their union, run over all different pairs, 
 and iterate until saturate.
We call {\em classes} this list of subsets.

\item From the data {\em classes}, choose representatives 
(if the list of subsets is ordered and their members are ordered, 
just pick the first member for each element of the list). Write
 down {\em all} cocycle 
equations, in terms of these representatives.
\item Eliminate the trivial equations, and
\begin{itemize}
\item for any cocycle equation where 1 appears, in case one found $a.1$, replace it
 by $1.a$, so we do not
count twice the same equation.
\item For any cocycle equation  of the form 
  $ac=bc$ or $ca=cb$, add
 $a\sim b$ and recalculate the equivalence relation that it 
generates.
\end{itemize}
With the new data {\em classes}
go to step 3, and iterate the process until it stabilizes. 
\end{enumerate}

The set in {\em classes}
containing 1 is called $S$, this is a list of  trivial elements in $\Un$.
A set 
of representatives of the others element of {\em classes}
give a set of generators of  $\Un$. The remaining 
nontrivial 2-cocycle equations, 
written in terms of these representatives, give a set of relations.
This algorithm produce a relatively small set of generators, and all the 
relations between them. For instance, in the example above will produce
$\Unc=Free(a,b)/(ab=1)$.
We have implemented this algorithm in G.A.P, and  it produces
the following:

\

For the dihedral quandle $D_3$:

\noindent Set of  generators:
$\{
f_{1}=[1,2], f_{2}=[1,3], f_{3}=[2,1], f_{4}=[2,3], f_{5}=[3,1], f_{6}=[3,2]
\}$. 
Trivial elements $S$: $1=[1,1]=[2,2]=[3,3]$. Relations:
\[ 
f_{1}f_{5}=f_{2},\ f_{2}f_{3}=f_{1},\ f_{3}f_{6}=f_{4},\ f_{4}f_{1}=f_{3},\ f_{5}f_{4}=f_{6},\ f_{6}f_{2}=f_{5}
\]
For Wada, the set of generators is $\{
f_{1}=[1,2]=[1,3], f_{2}=[2,2]=[3,3]\}$. Trivial elements: 
$1=[1,1]=[2,1]=[2,3]=[3,1]=[3,2]$. Relations:
\[
f_{2}f_{1}=1\]
An important remark on notation: G.A.P. always gives a numbering of the
elements of its objects (and in particular, one cal always order them), 
for instance, $\Z/3\Z=\{\ra 0,\ra 1,\ra 2\}$ has
three elements, that G.A.P. number as 
$[1,2,3]$, where 1 is the first element, 2 is the second and so on; in order to 
identify the element one has to see the label, and in this case $1$ corresponds to 
$\ra 0$, $2$ corresponds to $\ra 1$ and $3$ to $\ra 2=\ra -1$, so the equation
$1=[1,1]=[2,1]=[2,3]=[3,1]=[3,2]$ means
\[
1=(\ra 0,\ra 0)=(\ra 1,\ra 0)=(\ra 1,\ra 2)=(\ra 2,\ra 0)=(\ra 2,\ra 1)
\]
The new element "1" that we add to $X\times X$ is called $[\ ]$, so for example
the equality $1=[1,1]=[2,1]=[2,3]=[3,1]=[3,2]$ comes from the fact that
the list (of lists) {\em classes} contains the element
$[[\ ],[1,1],[2,1],[2,3],[3,1],[3,2]]$.

In the following we use the notation as in G.A.P.

For the inverse solution to Wada's: generators $\{
f_{1}=[1,2]=[2,1]=[3,3], f_{2}=[1,3]=[2,2]=[3,1]\}$,
trivial elements: $1=[1,1]=[2,3]=[3,2]$, relations:
\[
f_{2}f_{2}=f_{1},\ f_{1}f_{1}=f_{2}
\]
For the flip on 2 elements $\{1,2\}$: generators $\{
f_{1}=[1,2], f_{2}=[2,1], 
\}$, trivial elements: $1=[1,1]=[2,2]$, and no equations at all. 

For the flip in 3 elements $\{1,2,3\}$: generators:
$\{
f_{1}=[1,2], f_{2}=[1,3], f_{3}=[2,1], f_{4}=[2,3], f_{5}=[3,1], f_{6}=[3,2]
\}$, trivial elements: $1=[1,1]=[2,2]=[3,3]$, relations:
$f_{2}f_{1}=f_{1}f_{2}$, $f_{4}f_{3}=f_{3}f_{4}$, $ f_{6}f_{5}=f_{5}f_{6}$.

Taking the list of biquandles
of cardinality 3 from Bartholomew and Fenn's list, adding the inverse
solutions (when they are not isomorphic), we obtain the table below.

We remark that the procedure gives not only the number of 
generators, but the full equivalence class, we omit 
the full data in the table just for space considerations.
We also add to the table the order of $\sigma$, and the number
of fixed points on the diagonal $\Delta:=\{(x,x):x\in X\}$, for instance,
$\Delta^\sigma=\Delta$ if $X$ is a quandle.

\[
\begin{array}{|c|c|c|c|c|c|}
\hline
name&\sigma& generators&equations&order&\#\Delta^\sigma\\
&&of\ \Unc && of\ \sigma& \\
\hline flip &BQ^3_1&6&f_{2}f_{1}=f_{1}f_{2}, f_{4}f_{3}=f_{3}f_{4},&2&3 \\
                    &  &  & f_{6}f_{5}=f_{5}f_{6}, && \\
\hline a\hbox{-}flip\cup\{1\} &BQ^3_2&3&f_{3}f_{2}=f_{2}f_{3},&2 & 1\\
\hline &BQ^3_3&3&-&4& 1\\
\hline &BQ^{3*}_3&3&-&4 &1 \\
\hline Wada(\Z_3)&BQ^3_4&2&f_{2}f_{1}=1, & 3&1 \\
\hline inv.\ Wada(\Z_3)&BQ^{3*}_4&2&f_{1}f_{1}=f_2,\ f_2f_2=f_1 & 3&1 \\
\hline &BQ^3_5&3&f_{2}f_{1}=f_{1}f_{2}, &2 & 3\\
\hline Q_3&BQ^3_6&3&-& 4& 3\\
\hline inverse\ Q_3&BQ^{3*}_6&3&-&4 &3 \\
\hline (x,y)\mapsto(\hbox{-}y,\hbox{-}x) &BQ^3_7&3&f_{3}f_{2}=f_{2}f_{3}, & 2& 1\\
\hline D_3&BQ^3_8&6&f_{1}f_{5}=f_{2}, f_{2}f_{3}=f_{1}, f_{3}f_{6}=f_{4},& 3&3 \\
                     & &  & f_{4}f_{1}=f_{3}, f_{5}f_{4}=f_{6}, f_{6}f_{2}=f_{5}, & 3& \\
\hline inverse\  D_3&BQ^{3*}_8&0&-,& 3&3 \\
\hline &BQ^3_9&2&f_{2}f_{2}=f_{1}, f_{1}f_{1}=f_{2}, &3 & 0\\
\hline &BQ^{3*}_9&0&- & 3& 0\\
\hline involutive (\Z_3)& BQ^3_{10}&2&f_{1}f_{2}=f_{2}f_{1}& 2&0 \\
\hline
\end{array}
\]

We remark that for some cases (i.e. $BQ_{4,8,9}$) the invariant $\Unc$ 
distinguish between $\sigma$ and $\ra\sigma$. For $BQ^3_3$, the generators are the same in  the strong sense that the equivalent classes
of generators (as equivalent classes in $X\times X$)
are the same, the relations are also the same (no relation at all),  so they will give the same knot/link invariants, even though $\sigma$ and $\ra\sigma$ are non isomorphic biquandle solutions.

For most of the cases there is no much more to say in order to describe
 $\Unc$ as a group, for instance $\Unc(flip)=\Unc(BQ^3_1)=\Z^2*\Z^2*\Z^2$:
 the free product of three copies of $\Z^2$,
 $\Unc(BQ^3_3)\cong
\Unc(BQ^3_6)\cong F_3$: the free group on 3 generators,
 $\Unc(BQ^3_2)\cong
 \Unc(BQ^3_5)\cong
 \Unc(BQ^3_7)\cong Free(a,b,c)/(bc=cb)$,
 $ \Unc(BQ^3_10)\cong \Z^2$,. On the other hand, 
  there are some simplifications for the remaining cases:
  
 \[
 \Unc(BQ^3_4)\cong Free(a,b)/(ab=1)=Free(a)\cong\Z
 \]
\[
 \Unc(BQ^3_9)\cong BQ^{3*}_4
\cong Free(a,b)/(a^2=b,\ b^2=a)
 \cong Free(a)/(a^3=1)
 \]
\[
\Unc(D_3)
\Unc(BQ^3_8)
=F_6/
(f_{1}f_{5}=f_{2}, f_{2}f_{3}=f_{1}, f_{3}f_{6}=f_{4},
 f_{4}f_{1}=f_{3}, f_{5}f_{4}=f_{6}, f_{6}f_{2}=f_{5})
 \]
Call $a:=f_1$, $b:=f_5$ and $c:=f_4$, we have
\[
ab=f_{2}, f_{2}f_{3}=a, f_{3}f_{6}=c,
 ca=f_{3}, bc=f_{6}, f_{6}f_{2}=b
\]
in particular, we can solve $f_2$, $f_3$ and $f_6$ in terms of $a$, $b$,
$c$, so $\Unc$ is generated by $a,b,c$. In order to know the relations,
we replace
$f_2=ab$, $f_{3}=ca$ and $f_6= bc$ in the above equations and get
\[
 abca=a, cabc=c, bcab=b
\]
or equivalently
\[
 abc=1, cab=1, bca=1
\]
whose solution is $c=(ab)^{-1}$.
We conclude $\Unc(D_3)=Free(a,b)$.

The computer program gives the set of generators and relations in a
reasonable human time for biquandles of cardinality 12 or less. As a matter
of numerical experiment, the groups associated to the inverse solution of
biAlexander solution on $\Z_m$, for $s=-1$, and $t=1$, are cyclic of order $m$
(in a non trivial way)
if $m=3$, $5$, $7$, $11$, $13$ (and much more complicated groups for $m=4,6,8,9,10,12$).
 We don\rq{}t know if this is a general fact
for all primes $p$.
There are nevertheless  some general results that can be prove without
computer:

\subsubsection*{Inverse quandle solutions}

If $(X,\t)$ is a quandle then $\sigma(x,y)=(y,x\t y)$
is a biquandle, and condition Unc2 is trivial. But 
if we consider the {\em inverse} solution:
$\ra\sigma(x,y)=(y\tt x, x)$
then condition (Unc2) is not trivial, we have the relations
\[
(x,y)(x,z)=(x,z\tt y)(x,y)
\]
 \[
(y\tt x, z\tt x)=(y,z)
    \]
\[
 (x,x)=1
\]
We see that, in presence of the second identity, the first one can be modified into
\[
(x,y)(x,z)=(x\t y,z)(x,y)
\hskip 1cm \mathrm{(Unc\ra Q)}\]
or also
$(x\t z,y\t z)(x,z)=(x\t y,z)(x,y)
$.
Notice that (Unc$\ra Q$), with $x=z$, says
\[
(x,y)(x,x)=(x\t y,x)(x,y)
\To
(x,y)=(x\t y,x)(x,y)
\To
1=(x\t y,x)\]
This equation, for $x=y$, gives $1=(x\t x,x)=(x,x)$.
That is, $1=(x\t y,x)$ implies the type I condition.
So, we may list a set of relations for $U_{nc}(\ra \sigma)$ in the following way

\[
\left\{
\begin{array}{rclc}
 (x,y)(x,z)  &=&(x\t y,z)(x,y) &(Unc\ra{Q}1)\\
(x\t y,x)    &=&  1           &(Unc\ra{Q}2)\\
(x\t z,y\t z)&=&(x,y)        &(Unc\ra{Q}3)
\end{array}
\right.
\]

\begin{coro}
\label{coro:inverso}
 Let $Q$ be a quandle 
 and consider the biquandle solution
 $\ra\sigma(a,b)=(b\tt a, a)$.
If $Q$ is such that  for every $z \in Q$ there exists $y$ with $z=x\t y$, then
$U_{nc}(\ra\sigma)=1$.
\end{coro}

\begin{proof}
Given $(z,x)$, let $y$ be such that $z=x\t y$, then
$ (z,x)=(x\t y,x)=1$.
\end{proof}

\begin{ex} If $(X,\t)=D_n=(\Z/n\Z, x\t y=2y-x)$
with $n$ is odd
 then $U_{nc}(\ra\sigma)=1$.
\end{ex}

 \section{The reduced $U_{nc}$}
 
 We recall that if $f:X\times X\to G$ is a (type I)  cocycle and
 $\gamma:X\to G$ is a function satisfying $\gamma(x)=\gamma(sx)$,
 then
 $f_\gamma(x,y):=\gamma(x)f(x,y)\gamma(\sigma^2(x,y))^{-1}$
 is also a (type I) 2-cocycle, and the knot/link invariant produce
 by $f$ is the same as the one produced by $f$. In particular,
 one can consider the universal 2-cocycle
  $\pi:X\times X\to \Unc$ and try to see if there is a cohomologous one,
  simpler that $\pi$. This procedure leads to a construction that we call
  {\em reduced} universal group:
  
  \begin{defi}
Let $\gamma:X\to U_{nc}$ be a (set theoretical) map such that $\gamma(x)=\gamma(s(x))$ and
 $\pi_\gamma:X\times X\to U_{nc}$  given by
 \[
 \pi_\gamma(x,y)= \gamma(x)(x,y)\gamma(\sigma^2(x,y))^{-1}
 \]
Define $S=\{(x,y)\in X\times X: \pi_{\gamma}(x,y)=1\in U_{nc}\}
\subseteq X\times X$ and consider the group
$\Ur$ defined by
\[
\Ur:=\Unc/<\pi(x,y)/(x,y)\in S>
\]
Denote $\overline{[x,y]}\in \Ur$ the class of $(x,y)$ and
$p:X\times X\to \Ur$ the map $p(x,y)=\overline{[x,y]}$.
  \end{defi}

 \begin{teo}\label{Ured}
With notations as in the above definition,
The map $p:X\times X\to \Ur$ has the following universal property:
 \begin{itemize}
 \item $p$ is a 2-cocycle.
 \item for any group $G$ and 2-cocycle $f:X\times X\to G$,
  there exists a cohomologous map $f^{\Gamma}$ and a group homomorphism
 $ f^{\Gamma}:\Ur\to G$ such that $f^{\Gamma}$ factorizes through $p$, that is
 $f^{\Gamma}=f^{\Gamma}p$.
 \end{itemize}
 \end{teo}
 \begin{proof}
 The fact that $p$ is a 2-cocycle is immediate.
  By
(\ref{universal property}) we obtain the existence of the unique group morphism $\overline{f}$ such that 
\[
 \xymatrix{
 X\times X\ar[d]_{\pi}\ar[r]^f&G\\
 U_{nc}\ar@{-->}[ru]_{\exists !\ \ra f}
 }\]
  commutes.
  Define $f^{\Gamma}:=\ra f\circ\pi_{\gamma}$;
  in diagram:
  $
 \xymatrix{
 X\times X\ar[r]^{\pi_{\gamma}}\ar[dr]_{f^{\Gamma}} & U_{nc}\ar[d]_{\overline{f}}\\
&G
 }$

 We have
 $$f^{\Gamma}(x,y)=\ra{f}\circ\pi_{\gamma}(x,y)
 =\ra{f}\circ\gamma(x)\ra{f}\circ \pi(x,y)\left(\overline{f}\circ\gamma(\sigma^{2}(x,y))\right)^{-1}
 $$
 so $f^{\Gamma}$ and $\overline{f}\circ\pi=f$ are cohomologous.
 Using again the universal property of $\Unc$, $f^\Gamma$ factorizes through $\Unc$, hence
 there exists a group homomorphism $\ra{f^\Gamma}:\Unc\to G$ such that
 $f^\Gamma=\ra{f^\Gamma}\circ\pi$.
 
 On the other hand, since $\pi_{\gamma}(S)=1$ we have $f^{\Gamma}(S)=\ra f(\pi_\gamma(S))=\ra f(1)=1$,
 but also
 $f^\Gamma(S)=\ra{f^\Gamma}(\pi(S))$, so the group homomorphism
 $\overline{f^{\Gamma}}:\Unc\to G$ induces a map
 $f^\gamma\Ur=\Unc/\pi(S)\to G$
such that, if  $p':U_{nc}\rightarrow U_{nc}/\pi(S)$ is the canonical group projection to the quotient ($p=p'\circ\pi$),
then $\ra {f^\Gamma}=f^\gamma\circ p'$. In diagram:
\[
  \xymatrix{
 X\times X\ar[dr]_{f^{\Gamma}}\ar[r]^\pi \ar@/^{2pc}/[rr]^p&   U_{nc}\ar[d]_{\exists! \ra{f^{\Gamma}}} \ar[r]^{\!\! p'}
 & \Unc/\pi(S)\ar@{-->}[ld]^{\exists f^{\gamma}}&\ar@{=}[l]\Unc^{\gamma}\\
 &G&
 }\]
 Clearly $f^{\Gamma}=f^{\gamma}\circ p$.

  \end{proof}
 For a given $\gamma$, the associated $\Ur$ is called the
 {\em reduced} universal group.

 \begin{coro}
Given a biquandle $X$, if there exists $\gamma:X\to\Unc$ such that $\Ur=1$ then every 2-cocycle
in $X$ is trivial.
 \end{coro}

 A general example of the above situation is given by some Alexander biquandles.
  \subsection{The Alexander biquandle}
 
Let $A=\Z[s,t,s^{-1},t^{-1}]$,
 $X$ an $A$-module
 and $\sigma:X\times X\to X\times X$ given by
 the matrix
\[
\left(\begin{array}{cc}
0&t\\
s&(1-st)
\end{array}\right)
\]
 equivalently $
  \sigma(x,y)=(sy,tx+(1-st)y)$.
The condition of being a fixed point is $x=sy$:
\[
 \sigma(sy,y)=((sy),t(sy)+(1-st)y)=(sy,y)
\]
 
Cocycle conditions are
 \[
 \left\{
 \begin{array}{rcl}
 (x,y)(tx+(1-st)y,z)&=&(x,sz)(tx+(1-st)sz,ty+(1-st)z)
\\
(sy, sz)&=&(y,z)
\\
(sy,y)&=&1
\end{array}
\right.\]

Following M. Gra\~na, we can adapt to the biquandle situation the proof for the quandle case
(see Lemma 6.1 of \cite{G}). Consider 
$\gamma:X\to U_{nc}$ given by 
$\gamma(x)=(0,cx)$,
where $c=(1+st)^{-1}$. Notice that $c$ is an endomorphism commuting with
$s$, and 
 $(sy,sz)=(y,z)\in U_{nc}$, so
 \[
\gamma(x)=(0,cx)=
(s0,scx)=
(0,csx)=
\gamma(sx)
\]
hence, we can use $\gamma$ in order to get another
2-cocycle,  cohomologous to $\pi$. Recall
\[
\pi_\gamma(x,y):= \gamma(x)(x,y)\gamma(\sigma^2(x,y))^{-1}
\]
where
$\sigma(x,y)=(sy,tx+(1-st)y)$, so
\[
\pi_\gamma(x,y)= (0,cx)(x,y)
 \big(0, c ( tx+(1-st)y ) \big)^{-1}
 \]
in particular
\[
\pi_\gamma(0,y)= (0,0)(0,y)(0,c(1-st)y)^{-1}
=(0,y)(0,y)^{-1}
=1
 \]
so, the class of $(0,y)$ is trivial  in $\Ur$.

\begin{lem}
Let $X$ be an Alexander birack such that $(1-st)$ is invertible in $\End(X)$. If we define $\gamma$ as above, then,
 the following identities hold in $\Ur$:
 \begin{enumerate}
 \item
$(x,0)=1$  for all $x$, and
\item $(a,b)=(a,b+a)$ for all $ a,b\in X$.
\end{enumerate}
\end{lem}
\begin{proof}
1. From the cocycle condition,
 \[
(x,y)(tx+(1-st)y,z)=(x,sz)(tx+(1-st)sz,ty+(1-st)z)
\]
taking  $x=0=z$ we get
\[(0,y)((1-st)y,0)=(0,0)(0,ty)
\]
but we know that $(0,*)=1$ in $\Ur$, so $((1-st)y,0)=1$,
and because $(1-st)$ is a unity we conclude  $(x,0)=1$ for all $x$.

\

2.  Using the cocycle condition 
 \[
(x,y)(tx+(1-st)y,z)=(x,sz)(tx+(1-st)sz,ty+(1-st)z)
\]
and clear $z$ from
$tx+(1-st)sz=0$, that is, set
$z=\frac{-t}{(1-st)s}x$, then
 \[
(x,y)\Big(tx+(1-st)y,\frac{-t}{(1-st)s}x\Big)=\Big(x,s\frac{-t}{(1-st)s}x\Big)(0,ty+(1-st)z)
\]
or
 \[
 (x,y)\Big(tx+(1-st)y,\frac{-t}{(1-st)s}x\Big)
 =\Big(x,\frac{-t}{(1-st)}x\Big) \hskip 1cm(*)
\]
clearing $y=\frac{-t}{1-st}x$, get
 \[
\Big(x,\frac{-t}{1-st}x\Big)(0,z)=(x,sz)\Big(tx+(1-st)sz,t\frac{-t}{1-st}x+(1-st)z\Big)
\]
or
\[
\Big(x,\frac{-t}{1-st}x\Big)
=(x,sz)\Big(tx+(1-st)sz,\frac{-t^2}{1-st}x+(1-st)z\Big) \hskip 1cm(\dag)
\]
in particular, using RHS of (*) = LHS of (\dag) with $y=sz$ we get
\[
(x,sz)\Big(tx+(1-st)sz,\frac{-t}{(1-st)s}x\Big)
=
(x,sz)\Big(tx+(1-st)sz,\frac{-t^2}{1-st}x+(1-st)z\Big)
\]
so
\[
\Big(tx+(1-st)sz,\frac{-t}{(1-st)s}x\Big)
=
\Big(tx+(1-st)sz,\frac{-t^2}{1-st}x+(1-st)z\Big)
\]
Now we simply change variables. Call $a=tx+(1-st)sz$, then
$(1-st)z=a-\frac{t}{s}x$, replacing 
\[
\Big(a,\frac{-t}{(1-st)s}x\Big)
=
\Big(a,\frac{-t^2}{1-st}x+a-\frac{t}{s}x\Big)
\]
or
\[
\Big(a,\frac{-t}{(1-st)s}x\Big)
=
\Big(a,\frac{-t}{(1-st)s}x+a\Big)
\]
Call $b=\frac{-t}{(1-st)s}x$ (notice that
 $(x,y)\mapsto(a,b)$ is bijective) and get
$
(a,b)
=
(a,b+a)
$.
\end{proof}
Inductively $(a,b)=
(a,b+na)\ \forall n\in \N$; if $a$ generates $X$ additively
then
\[
(a,b)
=
(a,0)=1 \ \forall b\in X
\]

\begin{coro}\label{corobialexp}
If $p$ is an odd prime,
 $X=\F_p$, and $s^{-1}\neq t\in \F_p\setminus\{0\}$, 
then every cocycle in the Alexander's birack in
 $X$ is cohomologous to the trivial one. In other words, the reduced Universal group $\Ur$ is trivial.
 In particular
every 2-cocycle in $D_3$ is trivial.\end{coro}
 
 \begin{rem}
 This generalizes the result of Gra\~na in \cite{G} where he proves the quandle case, that is, the case  $s=1$.
\end{rem}
 

A biquandle example that is not a quandle is the following:
\begin{coro}
Let $X=\Z_3$, then Wada's biquandle agree with biAlexander biquandle ($s=-1=t$, $1-st=-1\in\Z_3$),
so $\Ur=1$ and every non commutative 2-cocycle is trivial. In particular, for any coloring
with this biquandle, the corresponding element in
$\Unc$ is trivial.
\end{coro}
 
\begin{rem}
For the inverse solution $\ra\sigma$ of
of  Wada's biquandle (with $G=\Z_3$), the group $\Unc(\ra\sigma)=\langle a:a^3=1\rangle$ is not trivial. One may wander if there is a function
$\gamma$ such that $\Ur(\ra\sigma)$ is the trivial group, but this is not the case. We will see examples
where the invariant obtained using Wada's inverse solution is not trivial, actually, it distinguishes
the trefoil from its mirror image.
\end{rem}

\begin{rem}
In the previous corollary, the hypothesis $|X|$ being prime was essential, the smallest case where it fails is $X=\Z_4$, as an example of computation
 we calculate the  reduced universal group for $X=\Z_4$ and for $Z_8$
with  $s=-1$ and $t=1$.
\end{rem}
 
 \subsection{Reduced $\Unc$ in computer}
 
 It is clear that the procedure that computes $\Unc$ in the computer can be
 trivially adapted for the reduced version, just adding as input
 a given set $S_0\subset X\times X$, and begin with
 $S=S_0\cup \{(x,sx):x\in S\}$, instead of simply
 $S= \{(x,sx):x\in S\}$. The procedure will actually compute a list of generators
 and relations of the quotient
  group $\Unc/(S_0)$, so it could be also used to produce other quotients,
  not only $\Ur$. The advantage of $\Ur$ is that it gives the same
  knot/link invariant as $\Unc$, so in order to find suitable $S_0$'s one can
  do the following:
  
  \begin{itemize}
  \item In order to produce functions $\gamma:X\to \Unc$ with 
  $\gamma(x)=\gamma(sx)$, consider the equivalence relation on $X$ induced by $s$, that is
the equivalence relation generated by $x\sim s(x)$.
  Denote $\ra x$ the class of $x$ modulo $s$.
  \item for all pairs $(x,y)\in X\times X$, consider the 
   coboundary relation
\[
f_\gamma(x,y)=\gamma(x)f(x,y)\gamma(\sigma^2(x,y))^{-1}
\]
  if $\ra x=
\ra{\sigma^2(x,y)}$ then $f(x,y)$ is conjugated to
 $f_{\gamma}(x,y)$,
so $f(x,y)=1\iff f_{\gamma}(x,y)=1$, so it is clear that $(x,y)$ can
not be included in $S$ because of $\gamma$.
  \item if $\ra x\neq 
\ra{\sigma^2(x,y)}$ then we can choose $\gamma:X\to \Unc$ such that
$\gamma(z)=\gamma(sz)$ for all $z$ and
$\gamma(x)f(x,y)=\gamma(\sigma^2(x,y))$.
  \end{itemize}
 By the above considerations, it is useful to list 
 all tuples
 \[
( \ra x,(x,y),\ra{\sigma^2(x,y)})
\]
with $   \ra x\neq \ra{\sigma^2(x,y)}$. For any of these elements,
add the pair 
$( \ra x,\ra{\sigma^2(x,y)})$ to a set f ``used'' elements, so we continue
with the others with
$( \ra x,(x,y),\ra{\sigma^2(x,y)})$ with
 $   \ra x\neq \ra{\sigma^2(x,y)}$ but
 $ (  \ra x, \ra{\sigma^2(x,y)})$ not ``used''.
  This procedure  is
   easily implemented in G.A.P. For example, for the Dihedral quandle gives
 \[
  [0, [0,1 ],2 ],\  [ 0, [0, 2 ],1 ],\  [1, [1, 0 ],2 ] ]
\]
 so we can choose $\gamma(0)=1$, 
 $\gamma(2)= [0,1 ]$, $\gamma(1)=[0, 2 ]$ and hence
 define $S_0:=\{
  [0,1 ], [0, 2 ]\}$. With this entry, the procedure computing $\Ur$
  gives $S=X\times X$, that is $\Ur(D_3)$ is trivial,
in agreement with Corollary \ref{corobialexp}.
  We give the list of generators and relations
  of $\Ur$ for biquandles of cardinality 3, with the corresponding $S_0$.
  
  \[
\begin{array}{|c|c|c|c|
c|}
\hline
name&\sigma& generators&equations&\\
&&of\ \Ur && S_0\\
\hline flip &BQ^3_1&6&f_{2}f_{1}=f_{1}f_{2}, f_{4}f_{3}=f_{3}f_{4},&- \\
                    &&  &   f_{6}f_{5}=f_{5}f_{6},& \\
\hline a\hbox{-}flip\{2,3\}\cup\{1\} &BQ^3_2&3&f_{3}f_{2}=f_{2}f_{3},&-\\
\hline &BQ^3_3&3&-&-\\
\hline &BQ^{3*}_3&3&-&- \\
\hline Wada(\Z_3)&BQ^3_4&0&- & \{[1,2]\}
 \\
\hline inv.\ Wada(\Z_3)&BQ^{3*}_4&1&f_{1}^3=1 &-  \\
\hline &BQ^3_5&2&- &\{[1,3]\}\\
\hline Q_3&BQ^3_6&2&-&\{[2,1]\}\\
\hline inverse\ Q_3&BQ^{3*}_6&3&-& -\\
\hline (x,y)\mapsto(\hbox{-}y,\hbox{-}x) &BQ^3_7&3&f_{3}f_{2}=f_{2}f_{3}, & -\\
\hline D_3&BQ^3_8&0&-& \{[1,2],[1,3]\} \\
\hline inverse\  D_3&BQ^{3*}_8&0&-&-  \\
\hline &BQ^3_9&1&f_{1}^3=1, &-\\
\hline &BQ^{3*}_9&0&- & -\\
\hline involutive (\Z_3)& BQ^3_{10}&2&f_{1}f_{2}=f_{2}f_{1}&  -\\
\hline
\end{array}
\]

We exhibit another example of computation 
using  this algorithm: 
  
 \subsection{Reduced group for Alexander biquandle in $\Z_4$}
 
 In order to choose a possible $\gamma$, we list
as above elements of the form
 $(   \ra x,(x,y),\ra{\sigma^2(x,y)})$
 with  $(   \ra x,(x,y)\neq \ra{\sigma^2(x,y)})$.
 without repeating  $( \ra x,(x,y),\ra{\sigma^2(x,y)})$, this gives only one 
 element
 \[
   [ 0, [0, 1 ],2 ] ]
\]
so we compute $\Ur$ with $S_0=\{(0,1)\}$. The G.A.P. answer is

4 generators: $\{
f_{1}$=[2,1]=[4,1], $f_{2}$=[2,2]=[4,4], $f_{3}$=[2,3]=[4,3], $f_{4}$=[3,2]=[3,4]$\}$, the trivial elements are
\[
1=[1,1]=[1,2]=[1,3]=[1,4]=[2,4]=[3,1]=[3,3]=[4,2]
\]
and  relations
\[
f_{1}=f_{2}f_{3},\ f_{3}=f_{2}f_{1}, \
f_{1}f_{2}=f_{3},\
f_{3}f_{2}=f_{1}, 
\]
\[
 f_{2}f_{1}=f_{1}f_{2},\ f_{2}f_{3}=f_{3}f_{2},\ f_{3}f_{1}=f_{1}f_{3}
\]
Notice that $f_4$ do not appear in the list of relations.
Calling $a=f_2$, $b=f_3$, we get

\[
f_{1}=ab, \ b=af_{1},
 f_{1}a=b, \
  ba=f_{1}\]
  \[ 
  bf_{1}=f_{1}b, \
  af_{1}=f_{1}a, \ ab=ba\]
Replacing $f_1$ by $ab$ we get
\[
 b=aab,
 aba=b, \
  ba=ab\]
  \[ 
  bab=abb, \
  aab=aba,\]
whose solution is
\[
a^2=1, \ ab=ba
\]
so $\Ur=Free(a,b,f_4)/(a^2=1,ab=ba)\cong (\Z/2\Z\oplus\Z)*\Z$.

\subsection{Reduced Universal group of 4-cycles in $S_4$}
Another example of application of $\Ur$ is the following:
consider the quandle 
\[
Q=\{(1,2,3,4),(1,2,4,3),(1,3,2,4),(1,3,4,2),(1,4,2,3),(1,4,3,2)\}
\]
 that is, 4-cycles in $S_4$, with quandle operation
$x\t y=y^{-1}xy$. Recall that $f:Q\times Q\to \Unc(Q)$ 
is cohomologous to $f_\gamma$
if there exists a function $\gamma:Q\to \Unc$ such that
\[
f_\gamma(x,y)=\gamma(x)f(x,y)\gamma(x\t y)^{-1}
\]
If we list
$( x,(x,y),\sigma^2(x,y))$ without repeating ``used'' pairs $(x,x\t y)$, we get
\[
[ 1, [ 1, 2 ], 4 ], [ 1, [ 1, 3 ], 6 ], [ 1, [ 1, 4 ], 3 ], 
  [ 1, [ 1, 6 ], 2 ], [ 2, [ 2, 1 ], 6 ], [ 2, [ 2, 5 ], 4 ], 
 \]
 \[
  [ 2, [ 2, 6 ], 5 ], [ 3, [ 3, 1 ], 4 ], [ 3, [ 3, 4 ], 5 ], 
  [ 3, [ 3, 5 ], 6 ], [ 4, [ 4, 2 ], 5 ], [ 5, [ 5, 2 ], 6 ] 
  \]
If we define
$\gamma(1)=1$, $\gamma(4)=[1,2]$, $\gamma(6)=[1,3]$,
$\gamma(3)=[1,4]$,
$\gamma(2)=[1,6]$,
$\gamma(5)=\gamma(2)[2,6]$
then $S_0=
\{[1,2],[1,3],[1,4],[1,6],[2,6]\}$. If we compute $\Unc$ using our algorithm, it gives 30 generators with with 108 equations, while
$\Ur$ has only 5 generators with 20 equations
\[
1=f_{1}f_{3},
 f_{2}f_{4}=1,
 f_{3}f_{1}=1,
   f_{5}f_{1}=1,
    f_{1}f_{5}=1,
     \]
     \[
     f_{1}f_{1}=f_{2}, f_{1}f_{1}=f_{4},
      f_{1}f_{1}=f_{3}f_{5},
  f_{5}=f_{1}f_{4}, f_{1}=f_{2}f_{5}, 
  \]
\[  f_{1}=f_{3}f_{4},
  f_{2}f_{1}=f_{3}, f_{4}f_{3}=f_{1},
  f_{4}f_{1}=f_{5}, f_{5}f_{2}=f_{1},
  \]
  \[
  f_{1}f_{2}=f_{3},
   f_{1}f_{2}=f_{2}f_{1}, f_{1}f_{3}=f_{3}f_{1}, f_{1}f_{4}=f_{4}f_{1},
    f_{1}f_{5}=f_{5}f_{1},
\]
Call $a:=f_1$, then $f_3=f_5=a^{-1}$,
$f_2=f_4=a^2$, and replacing these values into the
20 equations, the only remaining condition
is $ a^4=1$, we conclude $\Ur=\langle a: a^4=1\rangle$

This quandle is interesting because it distinguish (using $\Ur$ and its canonical cocycle) the trefoil from its mirror image: there are 30 colorings,
6 of them give trivial invariant both for the trefoil and its mirror (these are the 6 constant colorings),
but the other 24 colorings gives $a^{-1}$ for the trefoil and $a$ for its mirror, and clearly $a\neq a^{-1}$ in $\langle a: a^4=1\rangle$.

\section{Some  knots/links and their n.c. invariants}

There are 3 quandles of size 3, {\em none of them} give nontrivial invariant
for knots up to 11 crossings. On the other hand,
using the biquandle $BQ^3_2$=aflip$\coprod\{1\}$,
from the list of 84 knots with less or equal to 10 crossings,
all of them  have exactly 3 different colorings, but there are 44
with nontrivial invariant.
For instance, figure eight has nontrivial invariant for tree biquandles
of size 3:
 $BQ^3_2$=aflip$\coprod\{1\}$, 
 $BQ^3_7$: $\sigma(x,y)$=$(-y,-x)$, and $BQ^3_9$.

We illustrate in next table the number of colorings
(denoted by $c$) and nontrivial invariants,
for knots up to 6 crossings
 and biquandles from Bartholomew and Fenn's list:

\[
\begin{array}{c|cc|cc|cc|cc|cc|cc|cc}
 &3_1&&4_1&&5_1&&5_2&&6_1&&6_2&&6_3\\
& c&& c&& c&& c&& c&& c&& c&\\
\hline
BQ^3_1  & 3  && 3 & & 3 & & 3 & & 3 & & 3  && 3  & \\
 BQ^3_2  & 3  && 3  & f_3, f_3& 3  & f_3^{-1}, f_3^{-1}& 3 & & 3 & & 3 & & 3  & f_3, f_3, \\
\hline
 BQ^3_3  & 3  && 3  && 3  && 3  && 3  && 3  && 3  & \\
 BQ^3_4  & 9  &&  1  && 1  && 3  && 9  && 3  && 1  & \\
 \hline
 BQ^3_5  & 3  && 3  && 3  && 3  && 3  && 3  && 3  & \\
 BQ^3_6  & 3  && 3  && 3  && 3  && 3  && 3  && 3  & \\
\hline
 BQ^3_7  & 3  && 3  & f_3, f_3& 3  & f_3^{-1}, f_3^{-1}& 3 & & 3 & & 3 & & 3  & f_3, f_3 \\
 BQ^3_8  & 9 & &  3 & & 3  && 3  && 9 & &3  && 3  & \\
\hline
 BQ^3_9  & 9  && 3  & f_1, f_1, f_1& 3  & f_1^{-1}, f_1^{-1}, f_1^{-1}& 3  && 
9  &&  3  && 3  & f_1, f_1, f_1 \\
 BQ^3_{10}  & 3  && 0  && 0  && 3  && 3  && 3  && 0  & \\
\end{array}
 \]
 We see that using only number of colorings we can separate this list of knots in 3 groups: $\{3_1,6_1\}$, $\{4_1,5_1,6_3\}$
and $\{5_2,6_2\}$, and  using the invariant 
we can also distinguish $5_1$ from all others. One interesting remark is that, using only quandles, there are always the trivial constant colorings,
but using biquandles it may happens that a knot admit no coloring at all, as
we see with biquandle $BQ^3_{10}$.

\subsection{Wada inverse of $\Z_ 3$}

Recall for $\ra\sigma$ the inverse solution
of Wada's for  $G=\Z_3$,
Unc=$\langle a : a^3=1\rangle$, the knot $7_4$  has 
nontrivial invariant under this biquandle solution
(trivial invariant for one coloring, $a$ and $a^2$ for the other two).
 This example
shows the importance of considering $\ra \sigma$
different from $\sigma$, since we have this example where
for $\sigma$=
Wada's solution for $\Z_3$,
every cocycle is coboundary and hence no invariant will appear,
but for the inverse solution $\ra\sigma$ we get nontrivial things.

\subsection{Alexander biquandle on $\Z_4$ and $\Z_8$}

The Borromean link has trivial linking number, but has only 3
colorings using $D_3$, so we distinguish from three 
separated unknots. The Unc invariant are trivial for all biquandles of size 3.

On the other hand, for the biAlexander
 biquandle on $\Z_4$ with $s=-1$ and $t=1$,
even though there are 64 colorings, they give non trivial invariants:

 Recall
Unc=$Free(a,b,f_4)/(b^2=1,ab=ba)$, the invariant for the Borromean link
is trivial in 40 colorings, but gives twice
$(\alpha,\alpha,1)$, $(\alpha,1,\alpha)$, $(1,\alpha,\alpha)$, $(1,\alpha,\alpha^{-1})$,
 $(\alpha,1, \alpha^{-1})$,
$(\alpha,\alpha^{-1},1)$ on the others, with $\alpha=a$ and $\alpha=a^ {-1}$.,

\

In a similar way, Whitehead's link has trivial linking number,
give trivial invariant for all biquandles of size 3 (even though
non-trivial number of colorings), with bialexander on $\Z_4$
also give trivial invariant, but with 
with biAlexander on $\Z_8$ one has non trivial invariants.
First we compute $\Ur$ for $\Z_8$, $t=1$, $s=-1$, with subset
$S_0=\{[1,2],[1,3],[2,2]\}$ (it may be seen that this is a subset corresponding
to a  convenient $\gamma$). The algorithm gives as answer that
$\Ur$ has 4 generators:
\[f_{1}=(2,1)=(2,7)=(4,1)=(4,7)=(6,1)=(6,3)=(8,1)=(8,3),
\]
\[ f_{2}=(2,3)=(2,5)=(4,3)=(4,5)=(6,5)=(6,7)=(8,5)=(8,7),
\]
\[ f_{3}=(2,4)=(2,6)=(4,2)=(4,4)=(6,6)=(6,8)=(8,4)=(8,6),
\]
\[ f_{4}=(3,2)=(3,4)=(3,6)=(3,8)=(7,2)=(7,4)=(7,6)=(7,8), 
\]
Trivial elements are
\[
1=[ 1, 1 ]=[ 1, 2 ]=[ 1, 3 ]=[ 1, 4 ]=[ 1, 5 ]=[ 1, 6 ]
=[ 1, 7 ]=[ 1, 8 ]=[ 2, 2 ]=[ 2, 8 ]=[ 3, 1 ]
\]
\[
=[ 3, 3 ]
=[ 3, 5 ]=[ 3, 7 ]=[ 4, 6 ]=[ 4, 8 ]=[ 5, 1 ]=[ 5, 2 ]
=[ 5, 3 ]=[ 5, 4 ]=[ 5, 5 ]=[ 5, 6 ]
\]
\[
=[ 5, 7 ]=[ 5, 8 ]
=[ 6, 2 ]=[ 6, 4 ]=[ 7, 1 ]=[ 7, 3 ]=[ 7, 5 ]=[ 7, 7 ]=[ 8, 2 ]=[ 8, 8 ]
\]
with relations 
\[
f_{1}f_{3}=f_{2},\  f_{3}f_{1}=f_{2},\    f_{3}f_{2}=f_{1},\ f_{3}f_{3}=1,
\]
\[
f_{2}f_{1}=f_{1}f_{2},\  f_{2}f_{2}=f_{1}f_{1}, \ f_{2}f_{3}=f_{1},\
 f_{3}f_{1}=f_{1}f_{3}, \ f_{3}f_{2}=f_{2}f_{3}
\]
Calling $a:=f_1$, $b:=f_3$, we get
\[
ab=f_{2}, f_{2}a=af_{2}, f_{2}f_{2}=aa, f_{2}b=a, ba=f_{2}, ba=ab, bf_{2}=a, bf_{2}=f_{2}b, bb=1\]
so $b^2=1$, and $f_2=ab=ba$. we conclude $\Ur=Free(a,b,f_4)/(b^2=1, ab=ba)$.

If we use this biquandle with Whitehead's link we get 64 colorings,
32 of them give trivial invariant, 16 colorings give $(b,1)$ and 16 colorings give as invariant $(1,b)$.

\section{Final comments}

In the examples we saw, very often the group $\Ur$ is non commutative, but 
we haven't found a knot/link with genuine non commutative invariant, that is, for
example a commutator of two non commuting elements of $\Ur$. Also, sometimes
$\Ur$ have pairs of commuting elements and other non commuting, for instance, $\Ur(biAlex(Z_8))=Free(a,b,f_4)/(b^ 2=1,ab=ba)$, but using this
biquandle, computing the invariants for knots and links with less than 11 crossings,
the elements $a$ and $b$ do not ``mix" with $f_4$. We don't know if this is a
 general fact or not, that is, if the invariant obtained is the same if 
we use the abelianization of $\Ur$.

If $\Ur$ happens to be abelian, then the information we get with the non commutative invariant
is essentially 
the state-sum invariant for the canonical cocycle $\pi_\gamma:X\times X\to \Ur$.
If this is the case (or if one consider the abelianization of $\Ur$), then our construction can be seen as a natural and nontrivial way to produce interesting 2-cocycles, so that sate-sum invariant becomes a procedure
with input only a biquandle, and not a biquandle {\em plus} a 2-cocycle, because a natural 2-cocycle is always present when one gives a biquandle.

Another natural question about state-sum invariant for biquandles
 is how to generalize it for 
 2-cocycles with values in nontrivial coefficients, which is known for quandles, but unknown for biquandles. In order
to answer this question, it should be
convenient to have an action of some group (to be defined) 
into the abelian group of coefficients where the 2-cocycle takes values, and if one imitates
the quandle case, one should define, for each crossing, an exponent (in this group)
that  twist the value of the cocycle at that crossing.
 If the exponent is well-define, that is, for instance it remains unchanged under 
 Reidemeister moves of other crossings,
then essentially it must be a non commutative 2-cocycle. The group
$\Un$ was the candidate, and in fact this was origin of the present work.
In the quandle case there is a natural map
$\Un(X)\to G_X$, where $G_X$ is the group generated by $X$ with
relations $xy=zt$ if $\sigma(x,y)=(z,t)$; the map
$\Un(X)\to G_X$ is simply determined by  $(x_1,x_2)\mapsto x_2$.
So, for quandles,  $G_X$-modules
are natural candidates for coefficients
(see \cite{CEGS}), or also $\Un(X)$-modules, or quandle-modules
as considered in \cite{AG}. We hope 2-cocycles with values in
$\Un$-modules will allow to define more general state-sum invariants,
 but at the moment we don't know how, we end remarking that 
for biquandles, there is no general well-defined map
$\Un(X)\to G_X$, and $\Un(X)$ sometimes is the trivial group.

In \cite{FG} one can found the GAP programs computing colorings, $\Ur$, and invariants
for knots and links given as planar diagrams.

\end{document}